\documentclass[11pt,a4paper]{amsart}

\usepackage[english]{babel}
\thanks{The author was supported by the project PRIMUS/20/SCI/002 from Charles University.}
\thanks{ \textsc{address:} Charles University, Faculty of Mathematics and Physics, Department of Algebra, Sokolov\-sk\' a 83, 18600 Praha~8, Czech Republic \\ \textsc{email:} abenes970@gmail.com}

\usepackage{url}
\usepackage{amsmath}        
\usepackage{amsfonts}       
\usepackage{amsthm}         
\usepackage{bbding}         
\usepackage{bm}             
\usepackage{fancyvrb}       
\usepackage[all]{xy}
\title{Counting extensions of imaginary quadratic fields}
\author[]{Alexandr Bene\v{s} }

\theoremstyle{plain}
\newtheorem{theorem}{Theorem}
\newtheorem{lemma}[theorem]{Lemma}
\newtheorem{definition}[theorem]{Definition}

\newcommand{\ok}{\mathcal{O}}
\newcommand{\pp}{\mathfrak{p}}
\newcommand{\mm}{\mathfrak{m}}
\newcommand{\ii}{\mathbb{I}}
\newcommand{\zz}{\mathbb{Z}}
\newcommand{\qq}{\mathbb{Q}}
\newcommand{\rr}{\mathbb{R}}
\newcommand{\cc}{\mathbb{C}}
\newcommand{\uu}{\mathcal{U}}

\begin{document}


\maketitle

\begin{abstract}
	The goal is to obtain an asymptotic formula for the number of quadratic extensions with bounded discriminant of a some quadratic number field with odd class number. This extends an already known result for $\qq$.
\end{abstract}

\section*{Introduction}

Let $K$ be a number field. We define the function $a_n$ as the number of quadratic extensions of $K$ with discriminant $\leq n $. In \cite{wood} it is shown that for $K$ the rational numbers the function $a_n$ has the asymptotic behaviour:
\begin{equation*}
a_\qq(n)= \frac{6}{\pi^2}n + o(n).
\end{equation*}
We will extend this theorem for $K$ an imaginary and real quadratic fields with odd class number. The main result for imaginary quadratic number fields is:
\begin{theorem}
	\label{main}
	For a quadratic imaginary number field $K=\qq(\sqrt{-p})$ with $p$ prime $ > 3$ that is $3$  $(mod$ $4)$, the function $a_K(n) = \#\lbrace L / K | \deg(L/K)=2, \mathfrak{d}_L \leq n\rbrace$ is asymptotically equal to
	\begin{equation*}
	a_K(n) = Cn + o(n)
	\end{equation*}
	where $C$ is given by
	\begin{equation*}
	C = \frac{1}{2\mathfrak{d}_K^2} \frac{ \mathrm{Res}_{s=1}\zeta_K(s)}{\zeta_K(2)} \prod_{\pp |2} \frac{g_\pp(1)}{(1+N(\pp)^{-1})} =\frac{\mathfrak{d}_K^{-5/2} \pi h_K}{2\zeta_K(2)} \prod_{\pp |2} \frac{g_\pp(1)}{(1+N(\pp)^{-1})}
	\end{equation*}
\end{theorem}
where $h_K$ is the class number, $\mathfrak{d}_K$ is the discriminant and  $g_\pp$ is a simple term that will be given exactly.  An analogous result for real number fields is Theorem \ref{mainreal}.\\
Our method will use class field theory, Similar method can be used to count cubic or higher degree extension or count extensions with certain splitting properties at a set of primes.
We will extend the proof found in \cite{wood} from the rational numbers to some quadratic number fields.  The book \cite{cohen} tells more about using class field theory to create algorithms to count extensions with certain properties. The article \cite{wright} uses this method for extensions of a number field with given Galois group. For more about characters on quadratic number fields see \cite{li}.\\
The main idea is to view quadratic extensions of $K$ as open index 2 subgroups of the absolute Galois group of $K$ and use the Artin reciprocity to transform them to subgrups of the group of id\`{e}les. These can also be seen as kernels of homomorphisms from id\`{e}les to $\zz/2\zz$. If $K$ has odd  class number, then all such homomorphisms can be constructed as sums of homomorphisms from local fields to $\zz/2\zz$ such that they send all units of $\ok_K$ to 0.\\
If the class number is even, then there are homomorphisms from the class group $Cl_K$ to $\zz/2\zz$, which would have to be taken to account. Also there would be more homomorphisms from local fields than homomorphisms from the id\`{e}les, because the group  $\mathrm{Ext}^1(Cl_K,\zz/2\zz)$ would be nonzero and to calculate it, we would need to know nontrivial information about the structure of the class group.\\
The discriminant of each extensions can be computed from the local homomorphisms. The information about the extension can be put into one function called the counting function. The asymptotic behaviour is then estimated from the counting function using a Tauberian theorem, if we rewrite the counting function using certain zeta and L-functions.

\section{Algebraic number theory}

We will denote $K$ a number field. $\ok_K$ its ring of integers, $\mathfrak{d}_K$ the discriminant (which will always be positive), $\pp$ a place of $K$ and $|-|_\pp$ the associated norm. Then $K_\pp$ will be the completion of $K$ with respect to the norm, a locally compact normed field called a local field associated to $\pp$. A local field is non-archimedean if it is associated to a discrete valuation.
For a non-archimedean local field $A$ we define the ring of integers $\ok_A = \lbrace x \in A : |x| \leq 1 \rbrace $, its maximal ideal $\mm = \lbrace x \in A : |x| < 1 \rbrace$ and the residue field $\kappa_A = \ok_A/\mm_A$ which is finite. The discreteness of the norm implies that the maximal ideal is principal and its generator is called uniformizer $\pi$. The absolute value is usually normalized so that the uniformizer has norm $q^{-1}$ where $q$ is the cardinality of the residue field. The units in $\ok_A$ are exactly the elements with $|a|=1$.
\\

\section{Groups of units of local fields}
For our main result we will need to know more about the structure of the group of units of local fields. For a local field $A$ we define the $n$-th unit group $\uu^{(n)}= 1 + \mm^n$ and $\uu=\ok^\times= \uu^{(0)}$. 
\begin{theorem}
	\label{mod}
	We have $A^\times \cong (\pi) \times \uu \cong (\pi) \times \uu /\uu^{(1)} \times \uu^{(1)}  $ and the quotients
	\begin{equation*}
	\uu/\uu^{(n)} \cong (\ok/\mm^n)^\times \text{ and } \uu^{(n)}/\uu^{(n+1)} \cong \uu/\mm
	\end{equation*}
	for $n \geq 1$.
\end{theorem}
\begin{proof}
	See 3.10 and 5.3 in Chapter II. in \cite{neukirch}.
\end{proof}

\begin{theorem} \label{log}
	For a non-archimedean local field $A$ of characteristic 0 there is an isomorphism
	\begin{equation*}
	\exp : \uu^{(n)} \to \mm^n
	\end{equation*}
	with the inverse $\log$ for $n >\frac{e}{p-1}$ where $p$ is the characteristic of the residue field and $e$ is the normalized valuation of $p$. The functions are given by the power series
	\begin{equation*}
	\exp(x) = \sum_{k=0}^\infty \frac{x^k}{k!}, \quad \log(1+x) = \sum_{k=0}^\infty \frac{(-1)^kx^k}{k}
	\end{equation*}
	which converge on their domain of definition.
\end{theorem}
\begin{proof}
	Propositition 5.5 in Chapter II: in \cite{neukirch}.
\end{proof}
We can put this together with the fact $\mm^n = (\pi^n)\uu \cong \ok \cong \zz_p^{[A/\qq_p]}$ as additive groups to get
\begin{theorem}
	\label{unit}
	For a local field $A$ of characteristic 0 and its ring of integers $\ok_A$ we have:
	\begin{align*}
	&	A^\times \cong   \zz \times \zz / (q-1)\zz \times \zz / (p^a)\zz \times \zz_p^d \\
	& \ok_A^\times \cong \zz / (q-1)\zz \times \zz / (p^a)\zz \times \zz_p^d
	\end{align*}
	where p is the characteristic of the residue field and q is its order, d is the degree of $A$ over $\qq_p$ and $a$ is some integer.
\end{theorem}
\begin{proof}
	Proposition 5.7 in Chapter II. in \cite{neukirch}.
\end{proof}

\section{Extensions of local fields}
For an extension of number fields $L/K$ and primes $\mathfrak{q}$ of $L$ above $\pp$ (this means that $\pp = \ok_K \cap \mathfrak{q}$), the local fields $L_\mathfrak{q}$ and $K_\pp$ form an extension with properties that tells us information about the primes.
For proofs of these theorems see Chapter 7 and 8 of \cite{ant}.
\begin{definition}
	If $A/B$ is an extension of non-archimedean characteristic 0 local fields with normalized valuations and uniformizers $\pi_A,\pi_B$,\\ then we define the degree of inertia as the degree of the extension of finite fields $f_{A/B} = [\kappa_A/\kappa_B]=[(\ok_A/(\pi_A))/(\ok_B/(\pi_B))]$ and the ramification index as $e_{A/B}=\nu_A(\pi_B)$. An extension is \emph{unramified} if the degree of ramification is 1.
\end{definition}
\begin{theorem}[Local field extensions] \label{extension}
	For an extension of number fields $L/K$ and nonzero primes $\pp \subset \ok_K$ and $\mathfrak{q} \subset \ok_L$ such that $\mathfrak{q} \cap \ok_K=\pp$ we have an extension of local fields $L_\mathfrak{q}/K_\pp$. If $\mathfrak{q}_i$ are all the primes of $L$ above $\pp$, then $L \otimes_K K_\pp \cong \prod L_{\mathfrak{q}_i}$. Furthermore the degree of inertia of the local field extension is equal to the degree of inertia of the primes and the same for the ramification index.
\end{theorem}
\begin{proof}
	See Proposition 8.2 in \cite{ant}.
\end{proof}
\begin{theorem}
	\label{unram}
	Finite unramified extensions $L/K$ of non-archimedean local field $K$ correspond to finite extensions of the residue field $\kappa_K = \ok_K/(\pi_K)$ by sending $L \mapsto \kappa_L = \ok_L/(\pi_L)$ and if $\kappa_L = \kappa_K[X]/(f(X))$, then $L = K[X]/(\overline{f}(X))$ for some lift $\overline{f}(X)$ of $f(X)$ to $K[X]$.
\end{theorem}
\begin{proof}
	Proposition 7.50 in \cite{ant}.
\end{proof}
\section{Id\`{e}les}
We follow Chapter VI of \cite{neukirch}.
We have a number field $K$ and the set of its places $\pp$ and the corresponding local fields $K_\pp$. The group of id\`{e}les is an object $\ii_K$ that collects all the local fields into one. 
\begin{definition}
	The group of id\`{e}les is defined
	\begin{equation*}
	\ii_K = \hat{\prod}_{\pp \text{ places}} K_\pp^\times = \lbrace (a_\pp) \in \prod_{\pp \text{ places}} K_\pp^\times | \text{all but finitely many  } a_\pp \text{ are in } \uu_\pp \rbrace
	\end{equation*}
\end{definition}
We can put a topology on id\`{e}les by defining a system of neighbourhoods of 1 to be the sets:
\begin{equation*}
\prod_{\pp \in S} W_\pp \times \prod_{\pp\notin S} \uu_\pp
\end{equation*}
where $S$ is a finite set of places including the infinite ones and $W_\pp$ are systems of neighbourhoods of 1 in $K^\times_\pp$. Systems of neighbourhoods of other points are obtained by translations.
\\
From the inclusions $K \subset K_\pp$ we get the diagonal embedding $K^\times \to \ii_K$. The quotient $\ii_K / K$ is called the id\`{e}le class group $C_K$ and it inherits the quotient topology from $\ii_K$. There is a natural homomorphism $\ii_K \to J_K, (a_i) \mapsto \prod_{\pp \text{ finite}} \pp^{\nu_\pp(a_\pp)}$ to the group of fractional ideals. It is obviously surjective. This homomorphism factorizes to a surjection $C_K \to Cl_K$ since $K^\times$ maps exactly to principal fractional ideals.

For a set of places $S$ we have a subgroup of id\`{e}les $\ii^S_K = \prod_{\pp \in S} K_\pp^\times \times \prod_{\pp \notin S} \uu_\pp$. If $S$ is the set of infinite places we will denote it by $\ii^\infty_K$. 
\begin{theorem}
	There is an exact sequence
	\\
	\begin{equation*}
	\xymatrix{
		0 \ar[r] & \ii^\infty_KK^\times/K^\times \ar[r] & C_K \ar[r] & Cl_K\ \ar[r] & 0
	}
	\end{equation*}
\end{theorem} 
\begin{proof}
	See Proposition 1.3 in Chapter VI. in \cite{neukirch}.
\end{proof}
Note that by the second isomorphism theorem $\ii^\infty_K K^\times / K^\times \cong \ii^\infty_K/(\ii^\infty_K\cap K^\times)$.

Id\`{e}les can be used to conveniently formulate the main result of class field theory: the global Artin reciprocity. 
\begin{theorem}[Artin reciprocity]
	There is a continuous homomorphism from the id\`{e}le class group of a number field to the abelization of its absolute Galois group
	\begin{equation*}
	C_K \to \mathrm{Gal}(K)^{ab}
	\end{equation*}
	called the Artin map.
	Furthermore this homomorphism is surjective and its kernel is the largest connected component of $C_K$ that includes 1.
\end{theorem}
\begin{proof}
	Chapter V.5 of \cite{cft}.
\end{proof}
This powerful result tells us that abelian extensions of $K$ can be described by the id\`{e}le class group.

\section{Field extensions} \label{section}
In this section we will assume that $K$ is an imaginary quadratic number field and that the class number of $K$ is odd. 

We can use the Artin reciprocity to classify all quadratic extensions of an imaginary quadratic number field $K$. All quadratic extensions are abelian. By Galois theory they correspond to index 2 (necessarily normal) closed (equivalently open) subgroups of $G(K)$, the absolute Galois group of $K$. All finite extensions of $K$ correspond to closed subgroups of $G(K)$ with finite index. Its commutator $\Gamma = [G(K),G(K)]$ is a normal subgroup which is contained in all normal subgroups $H$ such that $G(K)/H$ is abelian. So all index 2 closed subgroups correspond to index 2 closed subgroups that lie in $G(K)^{ab} = G(K)/ \Gamma$ or equivalently continuous surjective homomorphisms that lie in  $\mathrm{Hom}(G(K)^{ab},\zz/2\zz)$. We will call $\mathrm{Hom}_c(G(K)^{ab},\zz/2\zz)$ the subgroup of continuous homomorphisms. Now we use the Artin reciprocity to transform this to homomorphisms to $\zz/2\zz$ from the id\`{e}le class group.

\begin{theorem}
	The inclusion $\ii^\infty_KK^\times/K^\times \to C_K $ induces an isomorphism  \\ $\mathrm{Hom}(C_K,\zz/2\zz) \to \mathrm{Hom}(\ii^\infty_KK^\times/K^\times,\zz/2\zz)$ for $K$ with odd class number (not necessarily imaginary quadratic).
\end{theorem}
\begin{proof}
	We have the exact sequence
	\begin{equation*}
	\xymatrix{
		0 \ar[r] & \ii^\infty_KK^\times/K^\times \ar[r] & C_K \ar[r] & Cl_K \ar[r] & 0.
	}
	\end{equation*}
	 We will apply the left exact functor $\mathrm{Hom}(-,\zz/2\zz)$ to the sequence and use the property of $\mathrm{Ext}$ functors:
	\begin{align*}
	\xymatrix{
		\mathrm{Hom}(Cl_K,\zz/2\zz) \ar[r] & \mathrm{Hom}(C_K,\zz/2\zz) \ar[r] &\\
		\ar[r] & \mathrm{Hom}(\ii^\infty_KK^\times/K^\times,\zz/2\zz) \ar[r] & \mathrm{Ext}^1(Cl_K,\zz/2\zz)
	}
	\end{align*}
	
	The edge terms are 0 as we will now show. Because $Cl_K$ is odd \\$	\mathrm{Hom}(Cl_K,\zz/2\zz)=0$. We have $\mathrm{Ext}^1(Cl_K,\zz/2\zz)  =0$ since \\ $\mathrm{Ext}^1(\zz/p\zz,A) \cong A/(p) $ for any group $A$.

\end{proof}

We need to restrict ourselves to continuous homomorphisms. The topology on $\ii_K$ induces the subset topology on $\ii_K^\infty$ and the quotient topology on  $\ii^\infty_K/(\ii^\infty_K\cap K^\times)$.
Note that we have  $\ii^\infty_KK^\times/K^\times \cong  \ii^\infty_K/(\ii^\infty_K\cap K^\times)$.
\begin{theorem}
	There is an isomorphism of continuous homomorphisms\\ $ \mathrm{Hom}_c (C_K,\zz/2\zz) \to \mathrm{Hom}_c (\ii^\infty_K/(\ii^\infty_K\cap K^\times),\zz/2\zz)$ for $K$ with odd class number, where $\mathrm{Hom}_c (A,B)$ is the group of continuous homomorphisms from $A$ to $B$.
\end{theorem}
\begin{proof}
From the previous theorem we have a bijection of all homomorphisms $\mathrm{Hom}(C_K,\zz/2\zz) \to \mathrm{Hom}(\ii^\infty_KK^\times/K^\times,\zz/2\zz)\cong \mathrm{Hom}(\ii^\infty_K/(\ii^\infty_K\cap K^\times),\zz/2\zz)$.
The inclusion map $\ii^\infty_K/(\ii^\infty_K\cap K^\times) \to C_K$ is continuous, so a continuous homomorphism maps to a continuous homomorphism. So there is a map $\mathrm{Hom}_c (C_K,\zz/2\zz) \to \mathrm{Hom}_c (\ii^\infty_K/(\ii^\infty_K\cap K^\times),\zz/2\zz)$ and it is injective. \\
If the preimage of $\chi \in \mathrm{Hom}_c (\ii^\infty_K/(\ii^\infty_K\cap K^\times),\zz/2\zz) \subset \mathrm{Hom} (\ii^\infty_K/(\ii^\infty_K\cap K^\times),\zz/2\zz)$ is $\nu \in \mathrm{Hom} (C_K,\zz/2\zz)$, it is enough to prove that the kernel of $\nu$ is open so that $\nu$ is continuous. Because $	\ii^\infty_K$ is open in $\ii_K$ (every point $a \in 	\ii^\infty_K$ has an open neighbourhood $a\prod_{\pp \text{ infinite}} K_\pp^\times \times  \prod_{\pp \text{ finite}} \uu_\pp$), we have that $\ii^\infty_K/(\ii^\infty_K\cap K^\times)$ is open in $C_K$ (its preimage in $\ii_K$ is $\ii_K^\infty$). The kernel $H$ of $\chi$ is open in $\ii^\infty_K/(\ii^\infty_K\cap K^\times)$ and thus in $C_K$. The kernel of $\nu$ includes $H$ and so it is a union of cosets of $H$ in $C_K$ and therefore open.
\end{proof}

\begin{lemma} \label{correspond}
	There is a bijection between quadratic extensions of any number field $K$ and continuous surjective homomorphisms $C_K \to \zz/2\zz$ given by the Artin map. 
\end{lemma}
\begin{proof}
	Quadratic extensions of $K$ correspond to open index 2 subgroups of $G(K)$. If $\ok$ is the connected component of 1 in $C_K$, then the Artin map gives an isomorphism $C_K/\ok \to G(K)$. This gives correspondence of open index 2 subgroups of $G(K)$ and $C_K/\ok$. These subgroups of $C_K/\ok$ can be seen as kernels of continuous surjective homomorphisms $C_K/\ok \to \zz/2\zz$. These correspond to continuous surjective homomorphisms $C_K \to \zz/2\zz$ since the connected component of 1 is always mapped to 0 by continuity.
\end{proof}

We will now look at $Hom(\ii^\infty_K,\zz/2\zz)$ with $\ii^\infty_K\cap K^\times$ in the kernel. The elements of $\ii^\infty_K\cap K^\times$ all generate the unit ideal. Therefore they are exactly the units of $\ok_K$ which are $\lbrace +1,-1 \rbrace$ if $K \neq \qq[i],\qq[e^{2\pi/3}]$.
\begin{lemma} \label{technic}
	For an imaginary quadratic number field $K$, continuous homomorphisms $\ii^\infty_K \to \zz/2\zz$ are finite sums of local homomorphisms $\uu_\pp \to \zz/2\zz$ for some distinct primes $\pp$ of $K$. More explicitely $\chi = \sum_\pp \chi_\pp$ where $\chi_\pp$ is $\chi$ composed with the inclusion $\uu_\pp = (1,\dots,\uu_\pp,\dots,1) \to \ii^\infty_K$.
\end{lemma}
\begin{proof}
	Composition with the inclusion $\uu_\pp = (1,\dots,\uu_\pp,\dots,1) \to \ii^\infty_K$ induces homomorphisms $\chi_\pp : \uu_\pp \to \zz/2\zz$.
	The kernel of the homomorphism $\chi : \ii^\infty_K \to \zz/2\zz$ is open, so there is a finite set $S$ of primes such that the kernel includes $\prod_{\pp \in S}  W_\pp \times \prod_{\pp \notin S} \uu_\pp $, so all but finitely many $\chi_\pp$ are trivial and $\chi(\dots,a_\pp, \dots) = \prod_\pp \chi_\pp(a_\pp)$. The local field at the infinite place $\cc^\times $ is always mapped to 0 (since every element in $\cc^\times$ has a square root), so the sum only includes finite primes. \\
	Conversely every finite sum of local homomorphisms gives a continuous homomorphism as we will show. Let $\chi= \sum_{\pp \in S} \chi_\pp$. As we will see in Lemma \ref{finite}, every local homomorphism $\chi_\pp:\uu_{\pp} \to \zz/2\zz$ has some subgroup $\uu_{\pp}^{(n_\pp)}$ in the kernel and if $S$ is the set of primes in the sum, then every element $g$ in the kernel has an open neighbourhood  $g\prod_{\pp \in S}  \uu_{\pp}^{(n_\pp)} \times \prod_{\pp \notin S} \uu_\pp $, so the kernel is open.
\end{proof}
All these theorems can be summarized as:
\begin{theorem} \label{galois}
	Quadratic extensions of an imaginary quadratic number field $K$ with odd class number are in bijection with finite sums $\chi: \ii_K^\infty \to \zz/2\zz$, that can be written as sums  $\chi=\sum_\pp \chi_\pp$ of homomorphisms $\chi_\pp:\uu_\pp \to \zz/2\zz$ with $\ii^\infty_K\cap K^\times$ in the kernel of $\chi$ and $\chi_\pp$ is $\chi$ composed with the inclusion $\uu_\pp = (1,\dots,\uu_\pp,\dots,1) \to \ii^\infty_K$.
\end{theorem}
\section{Conductors}
Now we need to know how can we compute the discriminant of a number field defined by such homomorphisms. We can do it using conductors. 
\begin{definition}
	For a homomorphism $\chi_\pp : \uu_\pp \to \cc^\times$ we define the local Artin conductor to be the smallest integer $f_{\chi_\pp}$ such that $\uu^{(f_{\chi_\pp})}_\pp \subset \ker \chi_\pp$. We define the global Artin conductor $\mathfrak{f}$ of a homomorphism $\chi : G(K)^{ab} \to \cc^\times$ as $\mathfrak{f} = \prod_{\pp} \pp^{\mathfrak{f}_{\chi_\pp}}$ where $\chi_\pp$ are the homomorphisms induced from $\chi$ by the map $\uu_\pp = (1,\dots,\uu_\pp,\dots,1) \to \ii_K \to C_K \to G(K)^{ab}$.
\end{definition}
We can interprete the group $\zz/n\zz$ as a subgroup of $\cc^\times$ generated by the $n$-th roots of unity.
In our case $\chi$ is a homomorphism $\chi :G(K)^{ab} \to \zz/2\zz \cong \lbrace +1,-1 \rbrace \subset \cc^\times $.
\begin{lemma} \label{finite}
	For a character $\chi : \uu_\pp \to \zz/n\zz$ the conductor is finite.
\end{lemma}
\begin{proof}
	We have to show that some $\uu_\pp^{(m)}$ is in the kernel of $\chi$. Since $\zz/n\zz$ is finite, the power $(\uu_\pp)^n$ is in the kernel. From Theorem \ref{log} there is an isomorphism $\log: \uu_\pp^{(k)} \cong (\pi)^k$ for $k$ larger than some constant $l$, where $\pi$ is the uniformizer of $K_\pp$. We find that $\uu_\pp^{(k)} \subset(\uu_\pp)^n$ since every element $\log(a) \in  (\pi)^k$ has an $n$-th root $\log(a)/n \in (\pi)^{k-\nu_\pp(n)} \cong \uu_\pp^{(k-\nu_\pp(n))}$ if we choose $k$ so that $k-\nu_\pp(n) > l$. Therefore all elements in $\uu_\pp^{(k)}$ are $n$-th powers, so $\uu_\pp^{(k)} \subset(\uu_\pp)^n$.
\end{proof}
If the character is a sum, the conductor is computable from the summands.
\begin{theorem} \label{sums}
	Let $\chi$ be a sum of local characters $\chi= \sum_\pp \chi_\pp$ over distinct primes, like in Theorem \ref{galois}. The conductor of $\chi$ is $\prod_\pp \pp^{f_\pp}$, where $f_\pp$ is the local conductor of $\chi_\pp$.
\end{theorem}
\begin{proof}
	The characters $\chi_\pp$ are obtained from $\chi$ by composition with the inclusion $\uu_\pp = (1,\dots,\uu_\pp,\dots,1) \to \ii_K$, so this holds by the definition of the conductor.
\end{proof}

The relationship between conductor and discriminant is given by the following formula.
\begin{theorem} \label{disc}
	For an abelian number field extension $L/K$ the relative discriminant $\mathfrak{d}_{L/K}$ is given by
	\begin{equation*}
	\mathfrak{d}_{L/K} = \prod_{\chi \in \mathrm{Char}(\mathrm{Gal}(L/K))} \mathfrak{f}_\chi 
	\end{equation*}
	where $\mathrm{Char}(G)$ is the set of characters of $G$, that is homomorphisms $G \to \cc^\times $. 
\end{theorem}
\begin{proof}
	Can be found in \cite{neukirch}, VII.11.9.
\end{proof}
If $L/K$ is a quadratic extension, $\mathrm{Gal}(L/K) \cong \zz/2\zz$ and there are only 2 characters with one trivial. The trivial character has conductor 1. If $\chi : G(K) \to \zz/2\zz$ is a continuous character and $L$ is the number field corresponding to the open subgroup $\mathrm{ker}(\chi)$ by the Galois correspondence, then it factors as $G(K) \to G(K)/\mathrm{ker}(\chi)=\mathrm{Gal}(L/K) \to \zz/2\zz$ and hence $\chi$ gives the nontrivial character on $\mathrm{Gal}(L/K)$ so we can compute its conductor locally and thus get the relative discriminant of $L/K$. The absolute discriminant of $L$ is then easily computable.
\begin{theorem} \label{discriminant}
	Given a tower of number fields $L/K/S$ then the relative discriminant of $L/S$ can be computed as
	\begin{equation*}
	\mathfrak{d}_{L/S} = N_{K/S}(\mathfrak{d}_{L/K}) \mathfrak{d}_{K/S}^{\deg L/K}	
	\end{equation*}
	where $N_{K/S}(-)$ is the ideal norm of $K/S$.
\end{theorem}
\begin{proof}
	See \cite{neukirch} III.2.10.
\end{proof} 
Let's calculate the conductor of local fields for our quadratic imaginary number field $K$. We will start with the odd primes.
\begin{lemma} \label{even}
	If $\pp$ is a prime of $K$ not above 2 (i.e. $\pp \cap \zz = (p)$ for an odd prime $p$), then $\uu_\pp \cong  \zz/(p^i-1)\zz \times \zz/p^a\zz \times \zz_p^d$ where $i$ is the degree of inertia of $\pp$ in $K/\qq$. The number $d$ is 1 if $\pp$ is split and 2 if inert or ramified and $a$ is some integer.
\end{lemma}	
\begin{proof}
	We know from Theorem \ref{extension} that $K_\pp$ is an extension of $\qq_p$ and the degree of inertia and ramification is the same as for $\pp$ in $K/\qq$. We will use Theorem \ref{unit}. From this theorem we know that $i$ is the degree of inertia. The number $d$ is the degree of $K_\pp/\qq_p$ is using Theorem \ref{extension} the ramification degree times the inertia degree.
\end{proof} 
We can see that there are only 2 local characters for primes not above 2:
\begin{lemma}
	If $\pp$ is a prime not above 2, then there are two characters on  $\uu_\pp$ one of which is trivial.
\end{lemma}
\begin{proof}
	There are no nontrivial homomorphisms to $\zz/2\zz$ from $\zz/p^a\zz$ since it has odd size and neither from $\zz_p$, since $\frac{1}{2} \in \zz_p$ because $\frac{1}{2} \in \qq_p$ and $\nu_p(\frac{1}{2})=0$ for odd $p$. Finally there is a nontrivial homomorphism from $ \zz/(p^i-1)\zz$ with kernel $ \frac{(p^i-1)}{2}\zz/(p^i-1)\zz$ since $2|(p^i-1)$.
\end{proof}
The conductors of these local characters are easily computable:
\begin{lemma} \label{cond}
	If $\pp$ is prime not above 2, then the local conductor of the trivial local character on $\uu_\pp$ is 0 and for the nontrivial one with kernel $ \frac{(p^i-1)}{2}\zz/(p^i-1)\zz$ it is 1.
\end{lemma}
\begin{proof}
	The trivial one has kernel $\uu_\pp = \uu_\pp^{(0)}$ and the nontrivial one factors through $ \uu_\pp^{(1)}$ since $\uu_\pp/ \uu_\pp^{(1)} \cong  \zz/(p^i-1)\zz$ by Theorem \ref{mod}.
\end{proof}

If $\pp$ is above 2, then we have three cases depending whether 2 is split, ramified or inert in $K$. Notice that 2 can't be ramified if $K$ is imaginary quadratic with odd class number unless $K=\qq(i)$ or $\qq(\sqrt{-2})$. 
\begin{theorem}
	The prime  2 is not ramified in the extension $K/\qq$, where $K$ is quadratic imaginary with odd class number unless $K=\qq(i)$ or $\qq(\sqrt{-2})$.
\end{theorem}
\begin{proof}
We will show that the ideal $2\ok_k$ cannot be written as a principal ideal squared.
If the only units in $\ok_K$ are +1,-1, then 2 is associated with a square only if $2=x^2$ or $-2=x^2$ for some $x \in \ok_K$. The first case is not possible since $K$ is imaginary and the second case is possible only in $\qq(\sqrt{-2})$. If there are more units then $K=\qq(i)$ or $K=\qq(\sqrt{-3})$ and for $\qq(\sqrt{-3})$ 2 is inert.
If $(2) = \mathfrak{q}^2$ for some non-principal ideal $ \mathfrak{q}$, then it has order 2 in the class group, which is a contradiction.
\end{proof}
\begin{lemma} \label{localfield}
	If 2 is split, we have  two primes $\pp$ and $ \mathfrak{q}$ above 2 and  $\uu_I \cong \zz_2^\times \cong  \zz/2\zz \times \zz_2$ for $I=\pp, \mathfrak{q}$.
	If 2 is inert, we have $\uu_2 \cong  \zz/3\zz \times \zz/2\zz \times \zz_2^2$.
\end{lemma}
\begin{proof} 
	We can use Theorem \ref{extension} to see that $K_I$ is isomorphic to $\qq_2$ if $I$ is split and $K_2$ is an unramified degree 2 extension of $\qq_2$ if it is inert.\\
	If 2 is split we can use Theorem \ref{unit} and the fact that $\ok_I$ is isomorphic to $\zz_2$. The exponent $a$ at $\zz/2^a\zz$ is 1 since $-1 \in \zz_2$, but $x^2=-1$ has no solution in $\zz_2$ (it doesn't have a solution in $\zz/4\zz$), so there is a second primitive root of unity in $\zz_2$, but no primitive forth root of unity.\\
	If 2 is inert we use the same theorem, $\ok_2$ is the ring of units of an unramified degree 2 extension of $\zz_2$.
	From Theorem \ref{unram} it can be  written as $\qq_2[X]/(f(X))$ where $f(X)$ is a polynomial such that $\zz/2\zz[X]/(f(X))$ is a degree two field extension of the local field $\zz/2\zz$, i.e. it is irreducible in $\zz/2\zz$. This polynomial is $f(X)=X^2+X+1$.
	So the local field is isomorphic to $\qq_2[X]/(X^2+X+1)$ and $\ok_2 \cong \zz_2[X]/(X^2+X+1)$ (the elements with valuation $\geq 0$). We also have $a=1$ since $x^2=-1$ doesn't have a solution in $\zz/4\zz[X]/(X^2+X+1)$, so there is no primitive forth root of unity.
\end{proof}

\begin{lemma} \label{multi4}
	There is an isomorphism  $(\zz/4\zz[X]/(X^2+X+1))^\times \cong (\zz/3\zz) \times (\zz/2\zz) \times (\zz/2\zz)$
\end{lemma}
\begin{proof}
	The invertible elements in $\zz/4\zz[X]/(X^2+X+1)$ are of the form $AX+B$ where both $A,B$ are not zero divisors in $\zz/4\zz$, that is 0 or 2. Therefore there are $4 \cdot 4-4=12$ elements, so the group is isomorphic to $\zz/3\zz \times \zz/2\zz \times \zz/2\zz$ or $\zz/3\zz \times \zz/4\zz$. The elements $2X+1,2X+3$ and $3$ have order 2 and so the group must be isomorphic to $\zz/3\zz \times \zz/2\zz \times \zz/2\zz$.	
\end{proof}

\begin{lemma} \label{multi8}
	There is an isomorphism  $(\zz/8\zz[X]/(X^2+X+1))^\times \cong (\zz/3\zz) \times (\zz/2\zz) \times (\zz/2\zz) \times (\zz/4\zz)$
\end{lemma}
\begin{proof}
	The invertible elements in $\zz/8\zz[X]/(X^2+X+1)$ are of the form $AX+B$ where both $A,B$ are not zero divisors in $\zz/8\zz$, that is 0,2,4 or 6. Therefore there are $8 \cdot 8-4 \cdot 4=48$ elements. Every element can be multiplied by one of $1,X,X^2$ so that it is of the form $2AX+B$. It can then be multiplied by one of $+1,-1$ so that $B$ is $1+4C$. These elements $2AX+4C+1$ form a subgroup with $4 \cdot 2=8$ elements and it has an element $2X+1$ of order 4 and 3 elements $5,4X+1,4X+5$ of order 2, so it is isomorphic to $\zz/2\zz \times \zz/4\zz$. 	
\end{proof}

It is now possible to determine all characters for even primes.
\begin{lemma} \label{conductor}
	For split primes over 2 we have 3 nontrivial characters with local conductors 2,3 and 3.
	If 2 is inert we have 7 nontrivial characters, 3 of them with local conductor 3 and 4 with local conductor 3.
\end{lemma}
\begin{proof}
	There are two homomorphism $\zz_2 \to \zz/2\zz$. One is trivial and the other one is nontrivial with kernel $2\zz_2$ (because $2\zz_2$ is always in the kernel and so the homomorphism is determined by the image of 1).\\
	For split primes we have $\uu_I \cong \zz/2\zz \times \zz_2 $  using Lemma \ref{localfield} so there are $2 \times 2=4$ characters of which 3 are nontrivial (with kernels $(0,2\zz_2),\\(\zz/2\zz,0),(\zz/2\zz,2\zz_2)$). Also $\uu_I/\uu_I^{(2)} \cong (\zz_2/4\zz_2)^\times \cong (\zz/4\zz)^\times \cong \zz/2\zz$ using Theorem \ref{mod}, so one of them has conductor 2 (a character factorizes through $\uu/\uu^{(n)}$ iff it has local conductor $\leq n$). We also have $\uu_I/\uu_I^{(3)} \cong (\zz/8\zz)^\times \cong \zz/2\zz \times \zz/2\zz$ and so there are 3 nontrival characters with local conductors $\leq 3$. One of them has local conductor 2, so the other two have local conductor 3.\\
	From the proof of Lemma \ref{localfield}, for inert prime 2 the local field is  isomorphic to \\ $\ok_2 \cong \zz_2[X]/(X^2+X+1)$. The group of units is $\uu_2 \cong (\zz_2[X]/(X^2+X+1))^\times \cong  \zz/3\zz \times \zz/2\zz \times \zz_2^2$ using Theorem \ref{unit}. 
	So there are 7 nontrivial characters total. Using Theorem \ref{mod} and \ref{multi4}, we can see that $\uu_2/\uu_2^{(2)} \cong (\zz/4\zz[X]/(X^2+X+1))^\times \cong \zz/3\zz \times \zz/2\zz \times \zz/2\zz$ so 3 of them have conductor 2 and $\uu_2/\uu_2^{(3)} \cong (\zz/8\zz[X]/(X^2+X+1))^\times \cong \zz/3\zz \times \zz/2\zz \times \zz/4\zz \times \zz/2\zz$ using Lemma \ref{multi8}, so all other characters factorize through $\uu_2^{(3)}$ and so have conductor 3.
\end{proof}

\section{Counting function}
We put all information about quadratic extensions of a number field $K$ into one function. We still assume that $K$ is imaginary quadratic with odd class number.\\ In this section we will assume that $K \neq \qq(i),\qq(\sqrt{-2}),\qq(\sqrt{-3})$.
\begin{definition}
	The counting function $f_K(s)$ of a number field $K$ is a function of a complex variable $s$ defined as the series
	\begin{equation*}
	f_K(s) = \sum_{n=0}^\infty a_n n^{-s}
	\end{equation*}
	where $a_n$ is the number of quadratic field extensions of $K$ with absolute discriminant $n$.
\end{definition}
We have already seen in Theorem \ref{correspond} that quadratic extensions correspond to nontrivial continuous homomorphisms $\ii^\infty_KK^\times/K^\times \cong  \ii^\infty_K/(\ii^\infty_K\cap K^\times)  \to \zz/2\zz$. If  $K \neq \qq(i),\qq(\sqrt{-3})$, we have $\ok_K^\times = \lbrace +1,-1 \rbrace$. Let's first look at homomorphisms $\ii^\infty_K \to \zz/2\zz$.
\\
In our case of an imaginary quadratic field we have $\ii^\infty_K = \cc^\times \times \prod_{\pp \text{ finite}} \uu_\pp$. We know from Theorem \ref{galois} that each character is a sum of a finite number of homomorphisms $\uu_\pp \to \zz/2\zz$.\\
We will see that the counting function can be written as a product of local factor over the primes. For convenience we will define the \emph{absolute conductor} of a character as the norm of the conductor of the character, which is natural number instead of an ideal.
\begin{definition}
	We define the local factor $g_\pp(s)$ at a prime $\pp$ to be the function
	\begin{equation*}
	g_\pp(s)= \sum_{\chi_i} N(\pp)^{-f_is}
	\end{equation*}
	where the sum is over all characters $\chi_i$ on $ \uu_\pp$ and $f_i$ is the local conductor of $\chi_i$
\end{definition}
From Lemma \ref{cond} we get that:
\begin{theorem} \label{gnot2}
	For primes $\pp$ not above 2 the local factor is
	\begin{equation*}
	g_\pp(s) =(1+N(\pp)^{-s})
	\end{equation*}
\end{theorem}
Similarly from Lemma \ref{conductor} we get
\begin{theorem} \label{local}
	If 2 is inert, then the local factor $g_2(s)$ is:
	\begin{equation*}
	g_2(s) = (1+3N(2)^{-2s}+4N(2)^{-3s})
	\end{equation*}
	and if 2 is split, then for the primes $I$ above 2 we have
	\begin{equation*}
	g_I(s)=(1+N(I)^{-2s}+2N(I)^{-3s}).
	\end{equation*}
\end{theorem}

First we show that a simplified counting function can be written as a product of local factors.
\begin{theorem} \label{ f0formula}
	The counting function $f_0(s) = \sum_{n \geq 1} a_n n^{-s}$ where $a_n$ is the number of continuous characters $\ii^\infty_K \to \zz/2\zz$ with absolute conductor $n$ can be written as
	\begin{equation*}
	f_0(s) =\prod_{\pp \text{ primes of }K}g_\pp(s) =\prod_{\pp |2} g(s)_\pp \times \prod_{\pp \neg |2} (1+N(\pp)^{-s}).
	\end{equation*}
	Furthermore it converges for $\mathrm{Re}(s)>1$.
\end{theorem} 
\begin{proof}
	Notice that the function $f_0(s)$ can also be written as a sum over ideals of $\ok_K$, that is $f_0(s)= \sum_{I \text{ ideal of }\ok_K} a_I N(I)^{-s}$, where $a_I$ is the number of homomorphisms with conductor $I$. This is the same sum, we are just indexing the terms by the conductors (ideals) instead of their norms (natural numbers).\\
	Using the theorems in Section \ref{section}, we know that every nontrivial homomorphism $\chi : \ii^\infty_K \to \zz/2\zz$ can be uniquely written as a sum of finitely many homomorphisms $\chi_\pp : \uu_\pp \to \zz/2\zz$ over distinct primes. Denote $D_k$ the set of primes of $\ok_K$ with norm less than $k$ and $H_k$ the set of ideals, that can be written as products of prime ideals from $D_k$. The set $D_k$ is finite, beacuse the are only finitely many ideals with norm less then some number (\cite{ant} Theorem 4.4). Every character $\chi$ whose conductor is in $H_k$ can be written as a sum of local characters over primes in $D_k$. This is because the absolute conductor of a sum of local homomorphisms over distinct primes $\sum_\pp \chi_\pp$ is $\prod_\pp N(\pp)^{f_\pp}$, by Theorem \ref{sums}.\\
	We have $\sum_{I \in H_k} a_I N(I)^{-s} = \prod_{\pp \in D_k} g_\pp(s)$ as we will show. The local factors $g_\pp(s)$ are sums of the terms $N(\pp)^{-f_\pp s}$ for all local characters $\chi_\pp$ on $\uu_\pp$ to the power $-s$ (here $f_\pp$ is the local conductor of $\chi_\pp$). If we  multiply out all the local factors for primes in $D_k$, we get exactly the sum of absolute conductors of all sums of local characters of primes in $D_k$ to the power $-s$. Therefore the product is $\sum_{\chi} N(I_\chi)^{-s}= \sum_{I \in H_k} a_I N(I)^{-s}$ where the first sum is over characters with conductor in $H_k$l and $I_\chi$ is the conductor of $\chi$.\\

	From this and the form of the factors $g_\pp(s)$, we can see that $a_I$ is at most 4, so the sum $f_0(s)= \sum_{I \text{ ideal of }\ok_K} a_I N(I)^{-s}$ converges for $\mathrm{Re}(s)>1$. We get the inequality
	\begin{equation*}
	|f_0(s)-  \prod_{\pp \in D_k} g_\pp(s)| \leq \sum_{I \notin H_k} a_I N(I)^{-\mathrm{Re}(s)}.
	\end{equation*}
	Thus on the halfplane $\mathrm{Re}(s)>1$ the product converges to $f_0(s)$ by letting $k$ go to infinity.
\end{proof}

The counting function looks like an Euler product for the Dedekind zeta function.

In fact, we have
\begin{theorem} \label{real}	
The function $f_0(s)$ can be expressed as
\begin{equation*}
f_0(s) =  \frac{\zeta_K(s)}{\zeta_K(2s)} \times \prod_{\pp |2} \frac{g(s)_\pp}{(1+N(\pp)^{-s})}
\end{equation*}
and therefore can be analytically extended to a holomorphic function for $\mathrm{Re}(s) > 1/2, s\neq 1$ with a simple pole at 1.
\end{theorem}
\begin{proof}
	We can write
	\begin{align*}
	f_0(s) = & \prod_{\pp |2} g(s)_\pp \times \prod_{\pp \neg |2} (1+N(\pp)^{-s})=\\
	= &  \prod_{\pp |2} \frac{g(s)_\pp}{(1+N(\pp)^{-s}} \times \prod_{\pp \text{ all primes}} (1+N(\pp)^{-s})=\\
	= &  \prod_{\pp |2} \frac{g(s)_\pp}{(1+N(\pp)^{-s}} \times \prod_{\pp \text{ all primes}} \frac{(1-N(\pp)^{-2s})}{(1-N(\pp)^{-s})}=\\
	=& \prod_{\pp |2} \frac{g(s)_\pp}{(1+N(\pp)^{-s}} \times \frac{\zeta_K(s)}{\zeta_K(2s)}.
	\end{align*}
	The function $\zeta_K$ is holomorphic outside 1 with a simple pole at 1. It also has no zeros for $\mathrm{Re}(s) >1$ so the function $\frac{1}{\zeta_K(2s)}$ is holomorphic on the halfplane $\mathrm{Re}(s) >1/2$. There are only 1 or 2 primes above 2, so the first factor is holomorphic on $\cc$.
\end{proof}
From Theorem \ref{galois} we know that quadratic extensions of $K$ correspond to homomorphisms from $\ii_K^\infty$ to $\zz/2\zz$ with $\ii^\infty_K \cap K^\times$ in the kernel. So far we have only looked at homomorphisms $\ii_K^\infty \to \zz/2\zz$, so we need to figure out which ones send $\ii^\infty_K \cap K^\times$ to $0 \in \zz/2\zz$. As we saw just before Lemma \ref{technic}, we have $\ii^\infty_K \cap K^\times=\lbrace +1,-1 \rbrace$. So the characters need to send $-1$ to 0. We will call a character \emph{even} is it does send $-1$ to 0 and \emph{odd} otherwise.

\begin{definition}
	We define the odd local factor $g_-(s)_\pp$ at a prime $\pp$ to be the function
	\begin{equation*}
	g_\pp(s)= \sum_{\chi_i} \sigma_i N(\pp)^{-f_is}
	\end{equation*}
	where the sum is over all characters $\chi_i$ on $ \uu_\pp$ and $f_i$ is the local conductor of $\chi_i$ and $\sigma_i$ is 1 if the character is even and $-1$ if it is odd.
\end{definition}
From Lemma \ref{cond} we get that:
\begin{theorem}
	For primes $\pp$ not above $2$ the odd local factor is
	\begin{equation*}
	g_-(s)_\pp =(1+N(\pp)^{-s})
	\end{equation*}
	if $\pp$ is inert or the prime $p$ (where $\pp$ is above $p$) is 1 mod 4 and
		\begin{equation*}
	g_-(s)_\pp =(1-N(\pp)^{-s})
	\end{equation*}
	if it is not inert and  $p$ is $3$ $(mod$ $4)$.
\end{theorem}
\begin{proof}
	Same as for the ordinary local factor, but now we have to look where is $-1$ sent by the nontrivial character on $\uu_\pp$. From \ref{even} we get that $\uu_K \cong  \zz/(p^i-1)\zz \times \zz/p^a\zz \times \zz_p^d$. The element $-1$ is mapped to  $(\frac{(p^i-1)}{2},0,0)$ by the isomorphism.
	If $\pp$ is inert, then $i=2$ and $4|(p^i-1)$, so $-1$ is mapped to 0 by the nontrivial character. This is also the case if $4|(p-1)$. Otherwise $-1$ is sent to $1 \in \zz/2\zz$ and the nontrivial character on $\uu_\pp$ is odd.
\end{proof}
We will write $C$ for the set of primes that satisfy the first condition in the previous theorem.\\
Similarly from Lemma \ref{conductor} we get
\begin{theorem} 
	If 2 is inert, then the odd local factor $g_-(s)_2$ is:
	\begin{equation*}
	g_-(s)_2 = (1+N(2)^{-2s}-2N(2)^{-2s}+2N(2)^{-3s}-2N(2)^{-3s})
	\end{equation*}
	and is 2 s split, then for the primes $I$ above 2 we have
	\begin{equation*}
	g_-(s)_I=(1-N(I)^{-2s}-N(I)^{-3s}+N(I)^{-3s}).
	\end{equation*}
\end{theorem}
\begin{proof}
	Same as for the ordinary local factor, but we use Lemmas \ref{multi4} and \ref{multi8} to figure out what characters are odd.\\
	If 2 is split, then -1 maps to $ 1 \in \zz/2\zz \cong (\zz/4\zz)^\times \cong \uu/\uu^{(2)}$ and to $(1,0) \in \zz/2\zz \times \zz/2\zz \cong (\zz/8\zz)^\times \cong \uu/\uu^{(3)}$, so one character of local conductor 2 is odd and two characters of local conductor at most 3 are odd, onee of them has local conductor 2. In total there is one odd character of local conductor 2 and one odd character of local conductor 3. Thus the local factor, denoted $g_-(s)_I$, is $(1-N(I)^{-2s}-N(I)^{-3s}+N(I)^{-3s})$.\\
	If 2 is inert,  -1 maps to $(0,1,0) \in \zz/3\zz \times \zz/2\zz \times \zz/2\zz \cong \uu/\uu^{(2)}$ and to 
	$(0,1,0,0) \in \zz/3\zz \times \zz/2\zz \times \zz/2\zz \times \zz/4\zz \cong \uu/\uu^{(3)}$ and we can see that two characters of conductor 2 are odd, and two characters of conductor 3 are odd, and thus the local factor is $(1+N(2)^{-2s}-2N(2)^{-2s}+2N(2)^{-3s}-2N(2)^{-3s})$.
	
\end{proof}

 \begin{theorem}
	The counting function $f_-(s) = \sum_{n \geq 1} b_n n^{-s}$ where $b_n$ is the number of continuous characters $\ii^\infty_k \to \zz/2\zz$ with absolute conductor $n$ that are even minus the number characters with conductor $n$ that are odd, can be written as
	\begin{align*}
	&f_-(s) =\prod_{\pp \text{ primes of }K} g_-(s)_\pp=\\ &=\prod_{\pp |2} g_-(s)_\pp \prod_{\pp \notin C, \pp  \neg | 2} (1+N(\pp)^{-s}) \prod_{\pp \in C} (1-N(\pp)^{-s})
	\end{align*}
	Furthermore it converges for $\mathrm{Re}(s)>1$.
\end{theorem}
\begin{proof}
	The proof is similar to the one for the counting function $f_0(s)$.
	If we write a character on $\ii_K^\times$ as a sum of local characters on $\uu_\pp$, than the character is even iff there is an even number of odd characters in the sum, since $-1$ is then mapped to 0 in $\zz/2\zz$. If we multiply out the product, the terms correspond to sums of local characters and the sign is positive if the sum has even number of odd local characters and negative if it has an odd number of odd local characters.

\end{proof}
Finally the counting fuunction of $K$ is expressed as:
\begin{theorem}
	The counting function $f_K(s)$ can be written as
	\begin{equation*}
	f_K(s) = \mathfrak{d}_K^{-2s} \frac{1}{2} ( f_0(s)+f_-(s) )-\mathfrak{d}_K^{-2s}
	\end{equation*}
	where $K$ is a quadratic imaginary number field with odd class number different from $\qq(i), \qq(\sqrt{-2}),\qq(\sqrt{-3})$ and $ \mathfrak{d}_K$ is the discriminant of $K$ as a natural number.
\end{theorem}
\begin{proof}

 We know from Theorem \ref{galois} that quadratic extensions of $K$ correspond to nontrivial continuous on $\ii_K^\infty \to \zz/2\zz$ with $(\ii_K^\infty \cap K^\times)$ in the kernel. We have $(\ii_K^\infty \cap K^\times)=\lbrace+1,-1 \rbrace$ for our $K$ and so the extensions correspond to even nontrivial continuous characters. Furthermore the absolute discriminant of the extension corresponding to the character $\chi$ is from Theorem \ref{discriminant} equal to $N(I_\chi)\mathfrak{d}
_K^2$ where $I_\chi$ is the conductor of $\chi$. \\
The function $f_0$ counts all characters, and the function $f_-$ counts even characters minus odd characters. Therefore adding them with a factor of one half counts only even characters. The terms in this counting function have terms $N(I)^{-s}$, so we have to multiply it by $\mathfrak{d}_K^{-2s}$ to get terms with the absolute discriminant. The term $-\mathfrak{d}_K^{-2s}$ is the eliminate the trivial character which correspond to the extension $K/K$ which we don't count.\\
\end{proof}
\section{Asymptotical distribution of number fields}
We can use the counting function to get the asymptotical distribution of quadratic extensions of $K$.\\
\begin{theorem} \label{tauberian}
	Let $f(s) = \sum_{n \geq 1} a_n n^{-s} $ be convergent for $ \mathrm{Re}(s) > a > 0$. Assume
	that in the domain of convergence $ f(s) = g(s)(s-a)^{-w} + h(s)$ holds, where $g(s)$,$ h(s)$ are
	holomorphic functions in the closed half plane $\mathrm{Re}(s) \geq a$, and $g(a) \neq 0$, and $w > 0$. Then
	\begin{equation*}
	\sum_{1 \leq n \leq X} a_n = \frac{g(a)}{a \Gamma(w)} X^a (\log X)^{w-1} + o(X^a(\log X)^{w-1})
	\end{equation*}
	As a special case, if $f(s)$ converges for $\mathrm{Re}(s) > 1$ and has meromorphic continuation to
	$\mathrm{Re}(s) \geq 1$ with a simple pole at $s = 1$ with residue $r$, then
	\begin{equation*}
	\sum_{1 \leq n \leq X} a_n = rX+o(X)
	\end{equation*}
\end{theorem}
\begin{proof}
	Corollary on page 121 of \cite{narkiewicz}
\end{proof}
We know that the function $f_0(s)$ is up to some simple factors equal to  $\frac{\zeta_K(s)}{\zeta_K(2S)}$. This function satisfies the special case of the previous theorem, we just have to compute the residue. The residue is $\mathrm{Res}_{s=1} \frac{\zeta_K(s)}{\zeta_K(2s)} = \frac{\mathrm{Res}_{s=1}\zeta_K(s)}{\zeta_K(2)}$. The residue of the Dedekind zeta function is given by the class number formula.
\begin{theorem} \label{residue}
	The residue of the Dedekind zeta function of the number field $K$ is
	\begin{equation*}
	\mathrm{Res}_{s=1}\zeta_K(s) = \frac{2^r_1(2\pi)^{r_2}\mathrm{Reg}_K h_K}{w_K \sqrt{\mathfrak{d}_K}}
	\end{equation*}
	where $r_1$ and $r_2$ is the number of real and complex places respectively, $\mathrm{Reg}_K$ is the regulator of $K$, $h_K$ is the class number and $w_K$ is the number of roots of unity in $K$.
	In particular, for an imaginary quadratic number field $r_1=0, r_2=1$ and $\mathrm{Reg}_K=1$
\end{theorem}
\begin{proof}
	See \cite{neukirch} VII. 5.11.
\end{proof}
The value $\zeta_K(2)$ for imaginary quadratic number field is calculated in \cite{zagier} Theorem 2.
\begin{equation*}
\zeta_K(2) = \frac{\pi^2}{6\sqrt{\mathfrak{d}_K}} \sum_{0<n<\mathfrak{d}_K} \left( \frac{-\mathfrak{d}_K}{n} \right) A \left( \cot(\frac{\pi n}{\mathfrak{d}_K}) \right)
\end{equation*}
where the function $A(x)$ is defined as
\begin{equation*}
A(x) = 2 \int_0^\infty \frac{t \text{d}t}{x \sinh^2t+ x^{-1} \cosh^2t}
\end{equation*}	
and $(\frac{a}{b})$ is the Kronecker symbol.\\

Now for the function $f_-(s)$. We can express it using Dirichlet L-functions. A Dirichlet character is a homomorphism $\chi :(\zz/m\zz)^\times \to \cc^\times$. We can extend it to $\zz$ by defining $\chi(a)=0$ if $\mathrm{gcd}(m,a) \neq 1$. For example we have a character $\chi_4 : (\zz/4\zz)^\times \to \cc$ where $\chi_4(3)=-1$. For an imaginary quadratic number field $K$ there is a character (the Kronecker symbol $(\frac{-\mathfrak{d}_K}{x})$) $\chi_K :(\zz/\mathfrak{d}_K\zz)^\times \to \cc$ such that $\chi_K(p)$ is 1 if $p$ splits in $K$ and -1 if $p$ is inert in $K$.\\

We will denote $I,S,R$ the set of inert, split and ramified primes of $K$ not above 2 and  $I_i,S_i,R_i$ $ i=1,3$ the subset of primes that are $i$ $(mod$ $4)$.
For every character $\chi : (\zz/m\zz)^\times \to \cc^\times$ we have an L-function

	\begin{equation*}
	L(\chi,s) = \prod_{p \text{ prime}} \frac{1}{1-\chi(p)p^{-s}}
	\end{equation*}
	which converges to a holomorphic function for $\mathrm{Re}(s)>1$.

If we ignore the factors from primes above 2, the function $f_-(s)$ is equal to 
\begin{align*}
& \prod_{\pp \notin C} (1+N(\pp)^{-s}) \prod_{\pp \in C} (1-N(\pp)^{-s}) =  \\
&= \prod_{p \in I} (1+p^{-2s}) \prod_{p \in S_1} (1+p^{-s})^2 
\prod_{p \in S_3} (1-p^{-s})^2 \times R(s)
\end{align*}
since the norm of a prime ideal $\pp$ above $p$ is $p^2$ if it is inert and $p$ otherwise and $p$ is odd in all products.
Here $R(s)$ is the factor for the (finitely many) ramified primes
$R(s) = \prod_{\pp \in R_1} (1+N(\pp)^{-s}) \times \prod_{\pp \in R_3} (1-N(\pp)^{-s})$.\\
\\

Now we can write:
\begin{align*}
&\prod_{p \in I} (1+p^{-2s}) \prod_{p \in S_1} (1+p^{-s})^2  
\prod_{p \in S_3} (1-p^{-s})^2 R(s) = \\
&= \frac{\prod_{p \in I} (1-p^{-4s}) \prod_{p \in S} (1-p^{-2s})^2 }{\prod_{p \in I} (1-p^{-2s}) \prod_{p \in S_1} (1-p^{-s})^2 \prod_{p \in S_3} (1+p^{-s})^2 }  \times\\
&\times \frac{\prod_{p \in R} (1-p^{-2s})}{\prod_{p\in  R_1} (1-p^{-s}) \prod_{p \in R_3} (1+p^{-s})}=\\
& = \frac{B(s)/\zeta_K(2s)}{\prod_{p \in I_1} (1-p^{-s})
	\prod_{p \in I_3} (1+p^{-s}) \prod_{p \in S_1 \cup R_1} (1-p^{-s}) \prod_{p \in S_3 \cup R_3} (1+p^{-s})} \\
& \times \frac{1}{\prod_{p \in I_1} (1+p^{-s})
	\prod_{p \in I_3} (1-p^{-s}) \prod_{p \in S_1} (1-p^{-s}) \prod_{p \in S_3} (1+p^{-s})}=  \\
& = \frac{L(\chi_4,s) L(\chi_4 \chi_K,s)}{\zeta_K(2s)} \times B(s)
\end{align*}
where $\chi_4 \chi_K$ is a character on $(\zz/\mathrm{lcm}(4,\mathfrak{d}_K)\zz)^\times$ and $B(s)= \prod_{\pp | 2 } 1/(1-N(\pp)^{-2s})$ is the factor of $\zeta_K(2s)$ at 2. For this character $\chi_4 \chi_K(p) = 1$ if $p$ is inert and 3 mod 4 or split and 1 mod 4 , $\chi_4\chi_K(p)=0$ is $p$ is ramified or 2 and $\chi_4 \chi_K(p)=-1$ otherwise.
We can summarize it in the following theorem.
\begin{theorem}
	\label{zeta}
	Let $K$ be a imaginary quadratic number field with odd class number not equal to $\qq[i],\qq[\sqrt{-3}],\qq[\sqrt{-2}]$.
	Then the function $f_-(s)$ can be written as 
	\begin{equation*}
	f_-(s) = \frac{L(\chi_4,s) L(\chi_4 \chi_K,s)}{\zeta_K(2s)} \times \prod_{\pp|2} g_-(s)_\pp \times B(s)
	\end{equation*}
	In particular, it is holomorphic for $\mathrm{Re}(s) \geq 1$.
\end{theorem}
\begin{proof}
	The equation is clear from the preceding discussion we just added the factors for the primes above 2. The Dedekind zeta function $\zeta(2s)$ is holomorphic and nonzero on the halfplane $\mathrm{Re}(s)>1/2$. Since the characters $\chi_4$ and $\chi_4 \chi_K$ are nontrivial if $K \neq \qq(i)$, the L-functions are holomorphic for all $s$.
\end{proof}
The conditions we have on the number field force it to be of the following form:
\begin{theorem}
    If $K$ is an imaginary quadratic number field with odd class number not equal to $\qq(i),\qq(\sqrt{-2})$ then it is of the form $\qq(\sqrt{-p})$ where $p$ is a prime that is $3$ $(mod$ $4)$.
\end{theorem}
\begin{proof}
        Because 2 is not ramified, $K$ is of the form $\qq(\sqrt{-D})$ for some $D = 3$ $(mod$ $4$) squarefree. By Gauss's genus theory (see chapter 4 and 6 in \cite{buell}) the class group contains exactly $2^{r-1}$ elements of order $\leq 2$ where $r$ is the number of primes dividing $D$. Therefore $D$ must be prime.
\end{proof}

We can put everything together to get the main result
\begin{theorem}
	\label{main}
For a quadratic imaginary number field $K=\qq(\sqrt{-p})$ with $p$ prime $ > 3$ that is $3$  $(mod$ $4)$, the function $a_K(n) = \#\lbrace L / K | \deg(L/K)=2, \mathfrak{d}_L \leq n\rbrace$ is asymptotically equal to
	\begin{equation*}
	a_K(n) = Cn + o(n)
	\end{equation*}
	where $C$ is given by
	\begin{equation*}
	C = \frac{1}{2\mathfrak{d}_K^2} \frac{ \mathrm{Res}_{s=1}\zeta_K(s)}{\zeta_K(2)} \prod_{\pp|2} \frac{g_\pp(1)}{(1+N(\pp)^{-1})} =\frac{\mathfrak{d}_K^{-5/2} \pi h_K}{2\zeta_K(2)} \prod_{\pp |2} \frac{g_\pp(1)}{(1+N(\pp)^{-1})}
	\end{equation*}
\end{theorem}
\begin{proof}
	Apply the special case of Theorem \ref{tauberian}  to the function \\$f_K(s) = \mathfrak{d}_K^{-2s} \frac{1}{2}(f_0(s)+f_-(s))-\mathfrak{d}_K^{-2s}$. This function is holomorphic for $\mathrm{Re}(s) \geq 1$ except for a pole at 1.
	The functions $f_-(s)$ and $\mathfrak{d}_K^{-2s}$ have no pole at 1, so it doesn't affect the asymptotic growth. The number $C$ is the residue of $f_K(s)$ at 1, which is thus $ \frac{1}{2\mathfrak{d}_K^2}$ times the residue of $f_0(s)$.  Then use Theorem \ref{residue} to get the formula for the residue.
\end{proof}
\section{Special cases}
We ignored the cases of $K = \qq(i),\qq(\sqrt{-3}), \qq(\sqrt{-2})$. We will show that the main Theorem \ref{main} holds for these fields in this form: \\
\begin{theorem}
		For a quadratic imaginary number field $K$ with odd class number the function $a_K(n) = \lbrace L / K | \deg(L/K)=2, \mathfrak{d}_L \leq n\rbrace$ is asymptotically equal to
	\begin{equation*}
	a_K(n) = Cn + o(n)
	\end{equation*}
	where $C$ is given by
	\begin{equation*}
	C = \frac{\mathfrak{d}_K^{-5/2} \pi h_K}{w_K \zeta_K(2)} \prod_{\pp |2} \frac{g_\pp(1)}{(1+N(\pp)^{-1})}
	\end{equation*}
	where $w_K$ is the number of roots of unity in $K$ and $g_\pp(s)$ is 
	
\begin{equation*}
	 g_\pp(s)= \left\{\begin{array}{lr}
		(1+2^{-2s}+2\cdot2^{-3s}), & \text{if 2 is split in $K$}\\
		(1+3\cdot4^{-2s}+4\cdot4^{-3s}), &\text{if 2 is inert in $K$}\\
		(1+2^{-2s}+2\cdot2^{-4s}+4\cdot2^{-5s}), & \text{if 2 is ramified in $K$}
	\end{array}\right \}
\end{equation*}

\end{theorem}
\begin{proof}
	We only need to show this for $K=\qq(i),\qq(\sqrt{-3}), \qq(\sqrt{-2})$.\\
The only problem with $\qq(\sqrt{-3})$ is that the unit group has six elements, $\ok_K^\times \cong \zz/6\zz \cong \zz/2\zz \times \zz/3\zz$. The $\zz/2\zz$ part is generated by $-1$ and the $\zz/3\zz$ is generated by the image of $\zeta_3 = e^{2\pi i/3}$. But $\zeta_3$ must map to $0 \in \zz/2\zz$ for every local character, since it has order 3, and so the Theorem \ref{main} works for the number field $\qq(\sqrt{-3})$, except now $w_K = 6$ in the formula for the residue of the Dedekind zeta function.\\
For the number field $\qq(\sqrt{-2})$ the problem is that 2 is ramified, which means that the local factor $g_{\sqrt{-2}}(s)$ is going to be different. We can calculate it similarly as in Lemmas \ref{conductor} and  \ref{local}. From Theorem \ref{extension}, we can see that at $\sqrt{-2}$ the local field is $\qq[X]/(X^2+2)\otimes_\qq \qq_2 \cong \qq_2[X]/(X^2+2)$ and its ring of integers is $\ok_{\sqrt{-2}} \cong \zz_2[X]/(X^2+2)$. Using Theorem \ref{unit} we get $\uu_{\sqrt{-2}} \cong \zz/2\zz \times \zz_2^2$, so there are also 7 non-trivial characters.
The uniformizer in the local field is not 2, but $\sqrt{-2}=X$ (it is the element with the lowest nonzero valuation and $\nu_{\sqrt{-2}}(2)=2$).
Using Theorem \ref{mod}, we can see that $\uu/\uu^{(n)} \cong (\zz_2[X]/(X^2+2,X^n))^\times$. By analyzing the structure of these groups, we can calculate the local factor to be $g_{\sqrt{-2}}(s)=(1+2^{-2s}+2\cdot2^{-4s}+4\cdot2^{-5s})$ .\\
For $\qq(i)$ there are two problems. The prime 2 is ramified ($(2) = (1+i)^2$) and the ring of integers of the local field is $\zz_2[X]/(X^2+1)$ and the uniformizer is $(1+i)=(1+X)$. We can compute the local factor as in the previous case and it is also $g_{1+i}(s)=(1+2^{-2s}+2\cdot2^{-4s}+4\cdot2^{-5s})$. \\
The unit group has 4 elements and is generated by $i$. So $i$ has to be mapped to 0 in $\zz/2\zz$ by the character, which means it has to be mapped to 1 by an even number of local characters. The number $i$ is an element of order 4. If $p$ is inert, that it is equal to 3 mod 4, then $8|p^2-1$ and so $i$ gets mapped to 0. If p is split, then $i$ is mapped to 0 iff $p$ is 1 mod 8. The counting function is constructed similarly, only now the function $f_-(s)$ is equal to 
\begin{equation*}
f_-(s) = \prod_{p \text{ ine.}} (1+p^2) \prod_{p \text{ spl., 1 mod 8}} (1+p)^2 \prod_{p \text{ spl., 5 mod 8}} (1-p)^2\times g_-(s)_{1+i} \times R(s),
\end{equation*}
 which can also be written as $L(\chi_8,s)L(\chi_4\chi_8,s)/\zeta_{\qq(i)}(2s) \times B(s)$ similarly as in Theorem \ref{zeta}, where the character $\chi_8(s)=-1$ for $s=5,7$ mod 8 and $\chi_8(s)=1$ otherwise. This function is holomorphic in the region $ \mathrm{Re}(s) \geq 1$ and so the Theorem \ref{main} holds even for $\qq(i)$.
\end{proof}
\section{The case of real number fields}
Let $K=\qq(\sqrt{D})$ be a real quadratic number field. The main difference with real number fields is that we have 2 infinite real places $\nu_1,\nu_2$. There is a nontrivial local character $\rr^\times \to \zz/2\zz$ which sends the negative numbers to 1. We also have more units, the unit group is generated by -1 and the fundamental unit $\varepsilon$.
We are thus looking for characters that send both -1 and $\varepsilon$ to 0. We will focus on the case where there is a unit of norm -1. This means that the negative Pell's equation $x^2-Dy^2=-1$ (when $D= 3$  (mod $4)$) or $x^2-Dy^2=-4$ (when $D= 1$ $($mod $4)$) has a solution. In this case the counting function $f$ is similar to the function $f_0$ in the imaginary case.
\begin{theorem}
For a real quadratic number field $K$ with odd class number such that 2 is not ramified and with a unit of norm -1. Then the counting function has the form:
\begin{equation*}
	f_0(s) =\prod_{\pp \text{ primes of }K}g_\pp(s)
	\end{equation*}
	where $g_\pp(s)$ is as in theorems \ref{gnot2} and \ref{local} .
	And the analogue of Theorem \ref{real} holds
	\begin{equation*}
f_0(s) =  \frac{\zeta_K(s)}{\zeta_K(2s)} \times \prod_{\pp |2} \frac{g(s)_\pp}{(1+N(\pp)^{-s})}.
\end{equation*}
\end{theorem}
\begin{proof}
    The fundamental unit $\varepsilon$ and -1 generate the units of $K$ by the Dirichlet unit theorem. The condition means that we have $\nu_1(\varepsilon) > 0, \nu_2(\varepsilon) < 0$. We also have  $\nu_1(-1) < 0, \nu_2(-1) < 0$. This means that in this case we can choose characters on the infinite real primes to send -1 and $\varepsilon$ to any combination of 0 and 1 in $\zz/2\zz$.
    The proof is then similar to Theorem \ref{ f0formula}, since for every character on the finite primes there is a unique character on the infinite primes such that their sum sends all units to 0. Therefore all characters can be chosen to be even.
\end{proof}
These number fields have the following form.
\begin{theorem}
    Let $K$ be a real quadratic number field with odd class number and with a unit of norm -1 such that 2 is not ramified.  Then these fields are exactly of the form $\qq({\sqrt{p}})$ where $p$ is a prime that is $1$ $(mod$ $4)$.
\end{theorem}
\begin{proof}
        $K$ must be of the form $\qq(\sqrt{D})$ for some squarefree $D$ that is 1 $mod$ 4. The Gauss's genus theory says that the narrow class group has exactly $2^{r-1}$ elements of order $\leq 2$ where $r$ is the number of primes dividing $D$. Because we have a unit of order -1, the narrow class group is the same as the class group. Therefore $D$ must be prime. Additionaly the Pell equation $x^2-Dy^2=-4$ has a solution ( see Theorem 9.3 in \cite{buell}), so the field has a unit of norm -1.
\end{proof}
From this we get the analogue of Theorem \ref{main}:
\begin{theorem} \label{mainreal}
	For a real quadratic number field $K=\qq(\sqrt{p})$ with $p$ prime that is $1$ $(mod$ $4)$ the function $a_K(n) = \#\lbrace L / K | \deg(L/K)=2, \mathfrak{d}_L \leq n\rbrace$ is asymptotically equal to
	\begin{equation*}
	a_K(n) = Cn + o(n)
	\end{equation*}
	where $C$ is given by
	\begin{equation*}
	    	C  =  \frac{1}{\mathfrak{d}_K^2 } \frac{\mathrm{Res}_{s=1}\zeta_K(s)}{\zeta_K(2)} \prod_{\pp |2} \frac{g_\pp(1)}{(1+N(\pp)^{-1})} =\frac{2\mathfrak{d}_K^{-5/2}  \log \varepsilon \cdot h_K}{\zeta_K(2)} \prod_{\pp |2} \frac{g_\pp(1)}{(1+N(\pp)^{-1})}
	    	\end{equation*}

\end{theorem}
\begin{proof}
    	Apply the special case of Theorem \ref{tauberian}  to the function \\$f_K(s) = \mathfrak{d}_K^{-2s} f_0(s)-\mathfrak{d}_K^{-2s}$ and then use the class number formula.
\end{proof}

\bibliographystyle{alphaurl}
\bibliography{knihy1}

\begin{thebibliography}{Woo14}

\bibitem[Bue89]{buell}
Duncan~A Buell.
\newblock {\em Binary quadratic forms}.
\newblock Springer, 1989.

\bibitem[Coh00]{cohen}
Henri Cohen.
\newblock {\em Advanced topics in computational number theory}.
\newblock Springer, 2000.

\bibitem[LR20]{li}
Wen-Ching~Winnie Li and Zeev Rudnick.
\newblock Pair arithmetical equivalence for quadratic fields.
\newblock \url{https://arxiv.org/abs/2007.13147}, 2020.

\bibitem[Mil13]{cft}
James Milne.
\newblock Class field theory, 2013.

\bibitem[Mil17]{ant}
James Milne.
\newblock Algebraic number theory, 2017.

\bibitem[Nar83]{narkiewicz}
W\l{}adys\l{}aw Narkiewicz.
\newblock {\em Number theory}.
\newblock World Scientific Publishing Co., 1983.

\bibitem[NS99]{neukirch}
J\"{u}rgen Neukirch and Norbert Schappacher.
\newblock {\em Algebraic number theory}.
\newblock Springer, 1999.

\bibitem[Woo14]{wood}
Melanie~Matchett Wood.
\newblock Asymptotics for number fields and class groups, 2014.
\newblock URL: \url{http://swc.math.arizona.edu/aws/2014/2014WoodNotes.pdf}.

\bibitem[Wri89]{wright}
David~J. Wright.
\newblock Distribution of discriminants of abelian extensions.
\newblock {\em Proc. London Math. Soc}, 1989.

\bibitem[Zag86]{zagier}
Don Zagier.
\newblock Hyperbolic manifolds and special values of dedekind zeta-functions.
\newblock 1986.

\end{thebibliography}

\end{document}